\documentclass{amsart}
\usepackage{amsmath}
\usepackage{amsfonts}

\newtheorem{theorem}{Theorem}
\theoremstyle{plain}

\newtheorem{corollary}{Corollary}

\newtheorem{lemma}{Lemma}

\newtheorem{proposition}{Proposition}
\newtheorem{remark}{Remark}

\numberwithin{equation}{section}
\input{tcilatex}

\begin{document}
\title[Fuzzification in AG-groupoids]{Fuzzy Abel Grassmann's Groupoids}
\author{Madad Khan and M. Nouman Aslam Khan}
\address{COMSATS Institute of Information Technology, Abbottabad.}
\email{madadmath@yahoo.com, mailnouman@gmail.com}

\begin{abstract}
In the present paper we have studied the concept of fuzzification in
AG-groupoids. The equivalent statement for an AG-groupoid to be a
commutative semigroup is proved. Fuzzy points have been defined in an
AG-groupoid and has been shown the representation of smallest fuzzy left
ideal generated by a fuzzy point. The set of all fuzzy left ideals, which
are idempotents, forms a commutative monoid. The relation of fuzzy
left(right) ideals, fuzzy interior ideals and fuzzy bi-ideals in AG-groupoid
has been studied. Necessary and sufficient condition of fully fuzzy prime
AG-groupoid has been shown. Further, It has been shown that the set of fuzzy
quasi-prime ideals of AG-groupoid with left identity forms a semillattice
structure. Moreover, equivalent statements for fuzzy semiprime left ideal in
an AG-groupoid have been proved.
\end{abstract}

\maketitle

\section{\textbf{Introduction}}

A fuzzy subset $f$ of a set $S$ is a function from $S$ to a closed interval $%
[0,1].$ The concept of a fuzzy subset of a set was first considered by L.A.
Zadeh [10] in 1965. A. Rosenfelt was the first who studied fuzzy sets in the
structure of groups [11]. Fuzzy semigroups were first considered by N.
Kuroki [12] in which he studied the bi-ideals in semigroups. The fuzzy
theory provides the underlying structure for the generalization of many
fields including logic, differential equations and groups. The fuzzy theory
on algebraic structures have been widley explored. An Abel Grassmann's
groupoid, abbreviated as AG-groupoid, is a groupoid $S$ whose elements
satisfy the left invertive law: $(ab)c=(cb)a$ for all $a,b$ and $c$ in $S$.
An AG-groupoid is the midway structure between a commutative semigroup and a
groupoid. It is a useful non-associative structure with wide applications in
theory of flocks. This property of an AG-groupoid inspired us to use the
fuzzy theory in AG-groupoids. In an AG-groupoid the medial law, $%
(ab)(cd)=(ac)(bd)$ for all $a,b,c$ and $d$ in $S$ holds [6]. If there exists
an element $e$ in an AG-groupoid $S$ such that $ex=x$ for all $x$ in $S$
then $S$ is called an AG-groupoid with left identity $e$. It is to be noted
that if an AG-groupoid $S$ has the right identity then $S$ is a commutative
monoid.If an AG-groupoid $S$ contains left identity then $(ab)(cd)=(dc)(ba)$
holds for all $a,b,c$ and $d$ in $S.$ Also $a(bc)=b(ac)$ holds for all $a,b$
and $c$ in an AG-groupoid with left identity. In an AG-groupoid $S$, an
element $a\in S$ is called idempotent if $a^{2}=a.$

Let $F(S)$ denote the collection of all fuzzy subsets of $S.$ For subsets $%
A,B$ of $S,$ $AB=\{ab\in S:a\in A,b\in B\}$. A non-empty subset $A$ of $S$
is called left(right) ideal of $S$ if $SA\subseteq A$ $(AS\subseteq A).$
Further $A$ is called two-sided ideal if it is both left and right ideal of $%
S$. A non-empty subset $A$ of $S$ is called interior ideal of $S$ if $%
(SA)S\subseteq A.$ A non-empty AG-subgroupoid $A$ of $S$ is called bi-ideal
of $S$ if $(AS)A\subseteq A.$ A non-empty subset $A$ of $S$ is called
idempotent if $AA=A.$ An ideal $P$ of $S$ is called prime if $AB\subseteq P$
implies that either $A\subseteq P$ or $B\subseteq P$, for all ideals $A$ and 
$B$ of $S.$ An AG-groupoid $S$ is called fully prime if every ideal is
prime. The left ideal $P$ of $S$ is called quasi-prime if $AB\subseteq P$
implies that either $A\subseteq P$ or $B\subseteq P$, for all left ideals $A$
and $B$ of $S.$ An AG-groupoid $S$ is called fully quasi-prime if every left
ideal is quasi-prime.

\section{$\protect\bigskip $Fuzzy Ideals in AG-groupoids}

Let $f$ and $g$ be two fuzzy subsets of an AG-groupoid $S.$ The product $%
f\circ g$ is defined by

$(fog)(x)=\QDATOPD\{ . {\underset{x=yz}{\bigvee }\left\{ f(y)\wedge
g(z)\right\} ,\text{\ if }\exists \text{ }y\text{ and }z\in S,\text{ such
that }x=yz,}{0\text{ \ \ \ \ \ \ \ \ \ \ \ \ \ \ \ \ \ \ \ \ \ otherwise.}}$

$\bigskip $

\begin{proposition}
Let $S$ be an AG-groupoid, then the set $(F(S),\circ )$ is an AG-groupoid.
\end{proposition}

\begin{proof}
Clearly $F(S)$ is closed. Let $f,g$ and $h$ be in $F(S).$ Then for any $x$
in $S$ we have, $((f\circ g)\circ h)(x)=\underset{x=yz}{\bigvee }\left\{
(f\circ g)\left( y\right) \wedge h(z)\right\} =\underset{x=yz}{\bigvee }%
\left\{ \underset{y=pq}{\bigvee }\left\{ f(p)\wedge g(q)\right\} \wedge
h(z)\right\} =\underset{x=(pq)z}{\bigvee }\left\{ f(p)\wedge g(q)\wedge
h(z)\right\} =\underset{x=(zq)p}{\bigvee }\left\{ h(z)\wedge g(q)\wedge
f(p)\right\} $. Then, further we have $\underset{x=wp}{\bigvee }\left\{ 
\underset{w=zq}{\bigvee }\left( h(z)\wedge g(q)\wedge f(p)\right) \right\} =%
\underset{x=wp}{\bigvee }\left\{ (h\circ g)\left( w\right) \wedge
f(p)\right\} =((h\circ g)\circ f)(x).$ Hence $(F(S),\circ )$ is an
AG-groupoid.
\end{proof}

\begin{corollary}
Let $S$ be an AG-groupoid, then the medial law holds in $F(S).$
\end{corollary}

\begin{proof}
Let $f,g,h,$ and $k$ be arbitrary elements of $F(S).$ By successive use of
left invertive law, $\left( f\circ g\right) \circ \left( h\circ k\right)
=\left( \left( h\circ k\right) \circ g\right) \circ f=\left( \left( g\circ
k\right) \circ h\right) \circ f=\left( f\circ h\right) \circ \left( g\circ
k\right) .$
\end{proof}

\begin{theorem}
Let $S$ be an AG-groupoid with left identity, then the following properties
hold in $F(S);$
\end{theorem}

$(i)\qquad f\circ \left( g\circ h\right) =g\circ \left( f\circ h\right)
\qquad $for all $f,g$ and $h\;in\;F(S),$

$(ii)\qquad \left( f\circ g\right) \circ \left( h\circ k\right) =\left(
k\circ h\right) \circ \left( g\circ f\right) \qquad $for all$\;f,g,h$ and $%
k\;in\;F(S).$

\begin{proof}
\textit{(i)} Let $x$ be an arbitrary element of $S$. If $x$ is not
expressible as a product of two elements in $S$, then $\left( f\circ \left(
g\circ h\right) \right) (x)=0=\left( g\circ \left( f\circ h\right) \right)
(x).$ Let there exists $y$ and $z$ in $S$ such that $x=yz,$ then$\left(
f\circ \left( g\circ h\right) \right) (x)=$ $\underset{x=yz}{\bigvee }%
\left\{ f\left( y\right) \wedge \left( g\circ h\right) (z)\right\} =\underset%
{x=yz}{\bigvee }\left\{ f\left( y\right) \wedge \underset{z=pq}{\bigvee }%
\left\{ g(p)\wedge h(q)\right\} \right\} =\underset{x=y(pq)}{\bigvee }%
\left\{ f\left( y\right) \wedge g(p)\wedge h(q)\right\} ,$ which implies
that $\underset{x=p(yq)}{\bigvee }\left\{ g\left( p\right) \wedge f(y)\wedge
h(q)\right\} =\underset{x=pw}{\bigvee }\left\{ g\left( p\right) \wedge 
\underset{w=yq}{\bigvee }\left\{ f(y)\wedge h(q)\right\} \right\} $ and so
we have $\underset{x=pw}{\bigvee }\left\{ g\left( p\right) \wedge \left(
f\circ h\right) (w)\right\} =\left( g\circ \left( f\circ h\right) \right)
(x).$ Thus, $\left( f\circ \left( g\circ h\right) \right) (x)=\left( g\circ
\left( f\circ h\right) \right) (x).$ If $z$ is not expressible as a product
of two elements in $S$, then$\left( f\circ \left( g\circ h\right) \right)
(x)=0=\left( g\circ \left( f\circ h\right) \right) (x).$ Hence, $\left(
f\circ \left( g\circ h\right) \right) (x)=\left( g\circ \left( f\circ
h\right) \right) (x)$ for all $x$ in $S$.

(\textit{ii)} If any element $x$ of $S$ is not expressible as product of two
elements in $S$ at any stage, then$\left( \left( f\circ g\right) \circ
\left( h\circ k\right) \right) (x)=0=\left( \left( k\circ h\right) \circ
\left( g\circ f\right) \right) (x).$ Let there exists $y,z$ in $S$ such that 
$x=yz,$ then$\left( \left( f\circ g\right) \circ \left( h\circ k\right)
\right) (x)=\underset{x=yz}{\bigvee }\left\{ \left( f\circ g\right)
(y)\wedge \left( h\circ k\right) (z)\right\} =\underset{x=yz}{\bigvee }%
\left\{ \underset{y=pq}{\bigvee }\left\{ f(p)\wedge g(q)\right\} \wedge 
\underset{z=uv}{\bigvee }\left\{ h(u)\wedge k(v)\right\} \right\} $ and
hence the above equality becomes $\underset{x=(pq)(uv)}{\bigvee }\left\{
f(p)\wedge g(q)\wedge h(u)\wedge k(v)\right\} =\underset{x=(vu)(qp)}{\bigvee 
}\left\{ k(v)\wedge h(u)\wedge g(q)\wedge f(p)\right\} =\underset{x=mn}{%
\bigvee }\left\{ \underset{m=vu}{\bigvee }\left\{ k(v)\wedge h(u)\right\}
\wedge \underset{n=qp}{\bigvee }\left\{ g(q)\wedge f(p)\right\} \right\} =%
\underset{x=mn}{\bigvee }\left\{ (k\circ h)(m)\wedge (g\circ f)(n)\right\}
=\left( (k\circ h)\circ (g\circ f)\right) (x).$
\end{proof}

\begin{proposition}
An AG-groupoid $S$ with $F(S)=(F(S))^{2}$ is commutative semigroup if and
only if $(f\circ g)\circ h=f\circ (h\circ g)$ holds for all fuzzy subsets $%
f,g$ and $h$ of $S.$
\end{proposition}

\begin{proof}
Let an AG-groupoid $S$ is commuatative semigroup. For any fuzzy subsets $f,g$
and $h$ of $S$ by use of proposition 1 and commutative law, $(f\circ g)\circ
h=(h\circ g)\circ f=h\circ (f\circ g).$ Conversally let $(f\circ g)\circ
h=f\circ (h\circ g)$ holds for all fuzzy subsets $f,g$ and $h$ of $S.$ We
have to show that an AG-groupoid $S$ is commutative semigroup. Let $f$ and $%
g $ be any arbitrary fuzzy subsets of $S.$ Since $F(S)=(F(S))^{2}$ so $%
f=h\circ k$ where $h$ and $k$ are any fuzzy subsets of $S.$ Now $f\circ
g=(h\circ k)\circ g=(g\circ k)\circ h=g\circ (h\circ k)=g\circ f$, which
shows that commutative law holds in $S.$ By successive use of proposition 1
and commutative law $(f\circ g)\circ k=(k\circ g)\circ f=f\circ (k\circ
g)=f\circ (g\circ k).$
\end{proof}

\bigskip

A fuzzy subset of $S$ is called fuzzy AG-subgroupoid of $S$ if $f(ab)\geq
f(a)\wedge f(b)$ for all $a$ and $b$ in $S,$ and is called fuzzy left(right)
ideal of $S$ if $f(ab)\geq f(b)$ \ ($f(ab)\geq f(a)$) for all $a$ and $b$ in 
$S.$ A fuzzy subset $f$ of $S$ is called a fuzzy two sided ideal(or a fuzzy
ideal) of $S$ if it is both fuzzy left and fuzzy right ideal of $S.$ A fuzzy
subset $f$ of $S$ is called fuzzy idempotent if $f\circ f=f.$ For a subset $%
A $ of $S$ the characteristic function, $C_{A}$ is defined by $%
C_{A}=\QDATOPD\{ . {1\text{\ if }x\in A,}{0\text{\ if }x\notin A.}$

Note that an AG-groupoid $S$ can be considered a fuzzy subset of itself and
we write $S=\mathbf{C}_{S},i.e.,$ $S(x)=1$, for all $x$ in $S.$ Let $a$ be
an arbitrary element of $S,$ then for $\lambda $ in $(0,1]$ and for $x$ in $%
S $ we define fuzzy point $a_{\lambda }$ of $S$ as; $a_{\lambda }\left(
x\right) =\QDATOPD\{ . {\lambda \text{\ \ if }x=a,}{0\text{ \ otherwise.}}.$
A very simple calculations shows that for a fuzzy subset $f$ of $S$ we have $%
f=\bigcup\limits_{a_{\lambda }\in f}a_{\lambda }.$ It is easy to see from
[1] that a non-empty subset $A$ of an AG-groupoid $S$ is AG-subgroupoid if
and only if $C_{A}$ is a fuzzy AG-subgroupoid of $S,$ and $A$ is left(right,
two-sided) ideal of $S$\ if and only if $C_{A}$ is a fuzzy left(right,
two-sided) ideal of $S.$ For non-empty sets $A$ and $B$ of an AG-groupoid $S$%
, $C_{A}\cap C_{B}=C_{A\cap B}$ and $C_{A}\circ C_{B}=C_{AB}.$ It is easy to
see that for every fuzzy subset $f$ of an AG-groupoid $S$, we have $%
f\subseteq S.$ The following lemmas have the same proof as in [1].

\begin{lemma}
Let $f$ be a fuzzy subset of an AG-groupoid $S,$ then the following
properties hold.

(i) $f$ is a fuzzy AG-subgroupoid of $S$ if and only if $f\circ f\subseteq
f. $

(ii) $f$ is a fuzzy left ideal of $S$ if and only if $S\circ f\subseteq f.$

(iii) $f$ is a fuzzy right ideal of $S$ if and only if $f\circ S\subseteq f.$

(iv) $f$ is a fuzzy ideal of $S$ if and only if $S\circ f\subseteq f$ and $%
f\circ S\subseteq f.$
\end{lemma}

\begin{lemma}
Let $S$ be an AG-groupoid. Then the following properties hold.

(i) Let $f$ and $g$ be two fuzzy AG-subgroupoids of $S.$ Then $f\cap g$ is
also a fuzzy AG-subgroupoid of $S.$

(ii)Let $f$ and $g$ be two fuzzy left(right, two-sided) ideal of $S.$ Then $%
f\cap g$ is also a fuzzy left(right, two-sided) ideal of $S.$
\end{lemma}

\begin{lemma}
In an AG-groupoid with left identity $S\circ S=S.$
\end{lemma}

\begin{proof}
Every $x$ in $S$ can be written as $x=ex,$ where $e$ is the left identity in 
$S.$ So $S\circ S(x)=\underset{x=yz}{\bigvee }\{S(y)\wedge S(z)\}\geq
\{S(e)\wedge S(x)\}=1.$ Hence $(S\circ S)(x)=1=S(x)$ for all $x$ in $S.$
\end{proof}

\begin{lemma}
\label{Sof=f}In an AG-groupoid $S$ with left identity, for every fuzzy left
ideal $f$ of $S$, we have $(S\circ f)=f.$
\end{lemma}

\begin{proof}
It is sufficient to show that $f\subseteq S\circ f$. Now for any $x$ in $S$, 
$(S\circ f)(x)=\underset{x=yz}{\bigvee }\{S(y)\wedge f(z)\}$. Since $x=ex,$
for all $x$ in $S,$ as $e$ is left identity in $S,$ so $\underset{x=yz}{%
\bigvee }\{S(y)\wedge f(z)\}\geq S(e)\wedge f(x)=f(x)$.
\end{proof}

\begin{proposition}
Let $S$ be an AG-groupoid with left identity and $f$ and $k$ are fuzzy left
ideals in $S$ then for any fuzzy subsets $g$ and $h$ of $S$, $f\circ
g=h\circ k$ implies $g\circ f=k\circ h.$
\end{proposition}

\begin{proof}
Since $f$ and $h$ are fuzzy left ideals in $S$ so by above lemma $S\circ f=f$
and $S\circ h=h.$ Now $g\circ f=(S\circ g)\circ f=(f\circ g)\circ S=(h\circ
k)\circ S=(S\circ k)\circ h=k\circ h.$
\end{proof}

\bigskip

The following corollary is direct consequence of the successive use of left
invertive law in fuzzy AG-groupoid shown in proposition 1.

\begin{corollary}
In an AG-groupoid $S$ for any fuzzy subsets $f,g$ and $h$ the following
conditions are equivalent:

(i) $(f\circ g)\circ h=g\circ (f\circ h),$

(ii) $(f\circ g)\circ h=g\circ (h\circ f).$
\end{corollary}

\begin{theorem}
If $S$ is an AG-groupoid then $Q=\{f\;|\;f\in S,\;f\circ
h=f\;where\;h=h\circ h\}$ is a commutative monoid in $S$.
\end{theorem}

\begin{proof}
The subset $Q$ is non-empty since $h\circ h=h$ which implies that $h$ is in $%
Q.$ Let $f$ and $g$ be fuzzy subsets of $S$ in $Q$ then $f\circ h=f$ and $%
g\circ h=g.$ Consider $f\circ g=(f\circ h)\circ (g\circ h)=(f\circ g)\circ
(h\circ h)=(f\circ g)\circ h$ which implies that $Q$ is closed. Now for any
fuzzy subsets $f$ and $g$ in $Q,$ $f\circ g=(f\circ h)\circ g=(g\circ
h)\circ f=g\circ f$ which implies that commutative law holds in $Q$ and
associative law holds in $Q$ due to commutativity. Since for any fuzzy
subset $f$ in $Q$ we have $f\circ h=f$, where $h$ is fixed, implies that $h$
is a right identity in $S$ and hence an identity.
\end{proof}

Let $S$ be an AG-groupoid and $a_{\lambda }$ be a fuzzy point in $S.$ The
smallest fuzzy left ideal of $S$ containing $a_{\lambda }$ is called a fuzzy
left ideal of $S$ generated by $a_{\lambda }.$ Let we denote the smallest
fuzzy left ideal of $S$ generated by $a_{\lambda }$ by $\left\langle
a_{\lambda }\right\rangle _{L}.$

\begin{theorem}
Let $S$ be an AG-groupoid with left identity $e$ and $a_{\lambda }$ be a
fuzzy point in $S,$ then $\left\langle a_{\lambda }\right\rangle _{L}=f$,
where $f(x)=\QDATOPD\{ . {\lambda \text{\ }\ \ \text{if \ there exists }b\in
S\text{ such that }x=ba,}{0\ \ \ \text{\ \ \ \ \ \ \ \ \ \ \ \ \ \ \ \ \ \ \
\ \ \ \ \ \ \ \ \ \ \ \ \ \ \ \ otherwise.}}$
\end{theorem}

\begin{proof}
Let $x$ and $y$ be an arbitrary elements of $S.$ Let $f(y)=0$ then $%
f(xy)\geq 0=f\left( y\right) $ and if $f\left( y\right) \neq 0$ then there
exists an element $b$ in $S$ such that $y=ba.$ Now, $%
f(xy)=f(x(ba))=f((ex)(ba))=f((ab)(xe))=f(((xe)b)a)=\lambda \geq f(y),$ which
implies that $f$ is a fuzzy left ideal in $S$. Since $a=ea,$ so $%
f(a)=f(ea)=\lambda $ and hence $a_{\lambda }$ is in $f.$ Let $g$ be any
other left ideal of $S$ containing $a_{\lambda },$ then $g(a)\geq \lambda .$
Let $f(x)=0$ for all $x$ in $S$ then $g(x)\geq 0=f(x).$ On the other hand if 
$f(x)=\lambda $ then there exists $b$ in $S$ such that $x=ba$ and $%
g(x)=g(ba)\geq g(a)\geq \lambda =f(x),$ which implies that $f\subseteq g.$
Hence $f$ is the smallest left ideal generated by $a$ in $S.$
\end{proof}

\begin{proposition}
Let $S$ be an AG-groupoid, then every fuzzy left ideal which is idempotent
is a fuzzy ideal.
\end{proposition}

\begin{proof}
Let $f$ be a fuzzy left ideal in $S$, which is idempotent. Consider, $f\circ
S=(f\circ f)\circ S=(S\circ f)\circ f\subseteq f\circ f=f.$
\end{proof}

\begin{remark}
If $S$ is an AG-groupoid with left identity, then in above proposition fuzzy
left ideal and fuzzy right ideal coincide.
\end{remark}

\begin{theorem}
Let $f$ is fuzzy idempotent in AG-groupoid $S$ with left identity, then
followings are true.
\end{theorem}

\textit{(i)} $S\circ f$ is an idempotent.

\textit{(ii)} Every fuzzy left ideal $g$ in $S$ commutes with $f.$

\begin{proof}
\textit{(i)} can easily be shown by use of corollary 1 and lemma 3.

For \textit{(ii)} consider, $f\circ g=\left( f\circ f\right) \circ g=\left(
g\circ f\right) \circ f\subseteq \left( g\circ S\right) \circ f\subseteq
g\circ f.$ Also, $g\circ f=g\circ (f\circ f)=f\circ (g\circ f)\subseteq
f\circ \left( g\circ S\right) \subseteq f\circ g.$
\end{proof}

\begin{theorem}
Let $S$ be an AG-groupoid with left identity, then the collection of all
fuzzy left ideals of $S$, which are idempotent forms a commutative monoid.
\end{theorem}

\begin{proof}
Let $\hat{H}$ denote the collection of all fuzzy left ideals which are
idempotent in $S$. Here, $\hat{H}$ is non-empty, since by lemma 3, $S\circ
S=S$ implies $S$ is in $\hat{H}$. Consider $f,g$ in $\hat{H}$, then $(f\circ
g)\circ (f\circ g)=(f\circ f)\circ (g\circ g)=f\circ g,$ also by corollary
1, $S\circ (f\circ g)=(S\circ S)\circ (f\circ g)=(S\circ f)\circ (S\circ
g)\subseteq f\circ g$. Also for every $f,g$ in $\hat{H}$ by use of theorem 1
and corollary 1, $f\circ g=(f\circ g)\circ (f\circ g)=(g\circ f)\circ
(g\circ f)=(g\circ g)\circ (f\circ f)=g\circ f,$ that is commutative law
holds in $\hat{H}$. Now, for any $f,g,$ and$\ h$ in $\hat{H}$ we have $%
(f\circ g)\circ h=(h\circ g)\circ f=f\circ (h\circ g)=f\circ (g\circ h).$
Since every $f$ in $\hat{H}$ is left ideal, so lemma 4 implies that $S\circ
f=f.$ Commutativity implies that $S\circ f=f\circ S=f,$ which implies that $%
S $ is identity in $\hat{H}$ and every $f$ is an ideal in $\hat{H}$.
\end{proof}

\begin{lemma}
Let $S$ be an AG-groupoid with left identity $e$. Then every fuzzy right
ideal is fuzzy ideal.
\end{lemma}

\begin{proof}
Let $f$ be a fuzzy right ideal in $S$, so $f\circ S\subseteq f.$ By use of
lemma 3 and proposition 1, $S\circ f=\left( S\circ S\right) \circ f=\left(
f\circ S\right) \circ S\subseteq f\circ S\subseteq f.$ So $f$ is a fuzzy
left ideal, and hence a fuzzy ideal in $S$.
\end{proof}

\begin{remark}
If $f$ is a fuzzy right ideal of an AG-groupoid $S$ with left identity then $%
f\cup (S\circ f)$ and $f\cup (f\circ f)$ are fuzzy two-sided ideals of $S.$
\end{remark}

\begin{lemma}
If $f$ is a fuzzy left ideal of an AG-groupoid $S$ with left identity then $%
f\cup (f\circ S)$ \ and $f\cup (f\circ f)$ are fuzzy two-sided ideals of $S.$
\end{lemma}

\begin{proof}
Consider, $(f\cup (f\circ S))\circ S=\left( f\circ S\right) \cup (\left(
f\circ S\right) \circ S)=(f\circ S)\cup ((S\circ S)\circ f)=(f\circ S)\cup
(S\circ f)=(f\circ S)\cup f=f\cup (f\circ S).$ Hence $f\cup (f\circ S)$ is
fuzzy right ideal of $S$, and by lemma 5, $f\cup (f\circ S)$ is fuzzy
two-sided ideal of $S.$ Now $(f\cup (f\circ f))\circ S=\left( f\circ
S\right) \cup \left( f\circ f\right) \circ S=(f\circ S)\cup ((S\circ f)\circ
f)\subseteq (f\circ S)\cup (f\circ f)=(f\circ f)\cup (S\circ f)\subseteq
(f\circ f)\cup f=f\cup (f\circ f),$ implies that $f\cup (f\circ f)$ is fuzzy
right ideal of $S$. By lemma 5, $f\cup (f\circ f)$ is fuzzy left ideal of $S$%
.
\end{proof}

\bigskip

A fuzzy subset $f$ of an AG-groupoid $S$ is called a fuzzy bi-ideal of $S$
if $f((xy)z)\geq f(x)\wedge f(z)$ for all $x,y$ and $z$ of $S$. It is easy
to see from [1] that for a non-empty subset $A$ of an AG-groupoid $S$ is a
bi-ideal of $S$ if and only if $C_{A}$ is a fuzzy bi-ideal of $S.$ The
following lemma has proof as in [1].

\begin{lemma}
Let $f$ be a fuzzy AG-subgroupoid of an AG-groupoid $S$. Then $f$ is a fuzzy
bi-ideal of $S$ if and only if $(f\circ S)\circ f\subseteq f.$
\end{lemma}

\begin{lemma}
Let $f$ and $g$ be fuzzy right ideals of an AG-groupoid $S$ with left
identity. Then $f\circ g$ and $g\circ f$ are fuzzy bi-ideals of $S.$
\end{lemma}

\begin{proof}
By use of corollary 1 $(f\circ g)\circ (f\circ g)=(f\circ f)\circ (g\circ
g)\subseteq f\circ g.$ Hence $f\circ g$ is fuzzy AG-subgroupoid of $S$. Now,
by proposition 1, lemma 3 and corollary 1 $((f\circ g)\circ S))\circ (f\circ
g)=((f\circ g)\circ (S\circ S))\circ (f\circ g)=((f\circ S)\circ (g\circ
S))\circ (f\circ g)\subseteq (f\circ g)\circ (f\circ g)\subseteq f\circ g.$
Similarly, $g\circ f$ is a bi-ideal.
\end{proof}

\begin{lemma}
Let $f$ and $g$ be fuzzy bi-ideals of an AG-groupoid $S$. Then $f\cap g$ is
a fuzzy bi-ideal of $S.$
\end{lemma}

\begin{proof}
It is same as in [1].
\end{proof}

\bigskip

A fuzzy subset $f$ of an AG-groupoid $S$ is called a fuzzy interior ideal of 
$S$ if $f((xa)y)\geq f(a)$ for all $x,a$ and $y$ of $S.$ It can easily seen
that, if $A$ be a non-empty subset of an AG-groupoid $S$, then $A$ is a
interior ideal of $S$ if and only if $C_{A}$ is a fuzzy interior ideal of $%
S. $ The following lemma has the proof as in [1].

\begin{lemma}
Let $f$ be a fuzzy AG-subgroupoid of an AG-groupoid $S$. Then $f$ is a fuzzy
interior ideal of $S$ if and only if $(S\circ f)\circ S\subseteq f.$
\end{lemma}

\begin{proposition}
Let $S$ be an AG-groupoid. Then for any fuzzy left ideal, which is
idempotent, in $S$, the following properties hold.
\end{proposition}

\textit{(i)} $f$ is a fuzzy bi-ideal.

\textit{(ii)} $f$ is a fuzzy interior ideal.

\begin{proof}
\textit{(i)} Since a fuzzy subset $f$ of $S$ is fuzzy left ideal so $f\circ
f\subseteq f.$ By use of corollary 1 $(f\circ S)\circ f=(f\circ S)\circ
(f\circ f)=(f\circ f)\circ (S\circ f)\subseteq f\circ f=f.$

\textit{(ii)} Consider, $(S\circ f)\circ S\subseteq f\circ S=(f\circ f)\circ
S=(S\circ f)\circ f\subseteq f\circ f=f,$ implies that $f$ is an interior
ideal of $S.$
\end{proof}

\begin{lemma}
Every fuzzy subset $f$ of an AG-groupoid $S$ with left identity is right
ideal if and only if $f$ is an interior ideal.
\end{lemma}

\begin{proof}
Let every fuzzy subset $f$ of $S$ is right ideal. For $x,a$ and $y$ of $S,$
consider $f((xa)y)\geq f(xa)=f((ex)a)=f((ax)e)\geq f(ax)\geq f(a),$ which
implies that $f$ is an interior ideal. Conversely, for any $x$ and $y$ in $S$
we have, $f(xy)=f((ex)y)\geq f(x).$
\end{proof}

\begin{lemma}
Let $f$ be a fuzzy left ideal in an AG-groupoid $S$ with left identity, then 
$f$ being interior ideal is bi-ideal of $S$.
\end{lemma}

\begin{proof}
Since $f$ is fuzzy left ideal in $S$, so $f(xy)\geq f(y)$ for all $x$ and $y$
in $S$. As $e$ is left identity in $S$. So, $f(xy)=f((ex)y)\geq f(x),$ which
implies that $f(xy)\geq f(x)\symbol{94}f(y)$ for all $x$ and $y$ in $S$.
Thus $f$ is fuzzy AG-subgroupoid. Consider, for any $x,y$ and $z$ in $S$, $%
f((xy)z)=f((x\left( ey\right) )z)=f((e(xy))z)\geq f(xy)=f((ex)y)\geq f(x)$.
Also $f((xy)z)=f((zy)x)=f((z(ey))x)=f((e(zy))x)\geq f(zy)=f((ez)y)\geq f(z).$
Hence $f((xy)z)\geq f(x)\symbol{94}f(z)$ for all $x,y$ and $z$ in $S$.
\end{proof}

\begin{proposition}
Let $f$ is a fuzzy subset of an AG-groupoid $S$ with left identity. If $f$
is a fuzzy left(right, two-sided) ideal in $S$ then $f\circ f$ is a fuzzy
ideal in $S.$
\end{proposition}

\begin{proof}
Let $f$ is a fuzzy left ideal in an AG-groupoid $S$, then by lemma 1 $S\circ
f\subseteq f.$ By use of lemma 3 and corollary 1, $S\circ \left( f\circ
f\right) =\left( S\circ S\right) \circ \left( f\circ f\right) =\left( S\circ
f\right) \circ \left( S\circ f\right) \subseteq f\circ f.$ Also by
proposition 1, $\left( f\circ f\right) \circ S=\left( S\circ f\right) \circ
f\subseteq f\circ f.$ If $f$ is a fuzzy right ideal in $S$ then by lemma 5 $%
f $ is a fuzzy left ideal.
\end{proof}

\begin{corollary}
Let $f$ is a fuzzy subset of an AG-groupoid $S$ with left identity. If $f$
is a fuzzy left ideal in $S$ then $f\circ f$ is a fuzzy bi-ideal and an
interior ideal in $S.$
\end{corollary}

\begin{proof}
By proposition 6, $f\circ f$ is a fuzzy ideal in $S$. Now by lemmas 11 and
12, $f\circ f$ is a fuzzy interior and fuzzy bi-ideal of $S.$
\end{proof}

\begin{theorem}
In an AG-groupoid $S$, every fuzzy ideal is a fuzzy bi-ideal and an interior
ideal of $S.$
\end{theorem}

\begin{proof}
Let $f$ be a fuzzy ideal of an AG-groupoid $S$. Clearly $f$ is
AG-subgroupoid of $S$ by lemma 1, since $f\circ f\subseteq S\circ f\subseteq
f.$ Consider, $\left( f\circ S\right) \circ f\subseteq f\circ f\subseteq f,$
which by lemma 7 shows that $f$ is a bi-ideal in $S.$ Now, consider $\left(
S\circ f\right) \circ S\subseteq f\circ S\subseteq S,$ which by lemma 10
shows that $f$ is an interior ideal in $S$.
\end{proof}

\bigskip

A fuzzy ideal $f$ of an AG-groupoid $S$ is called fuzzy prime ideal if for
any two fuzzy ideals $g$ and $h$ of $S$, $g\circ h\subseteq f$ implies that
either $g\subseteq f$ or $h\subseteq f$. An AG-groupoid $S$ is fully fuzzy
prime if every fuzzy ideal is prime in $S.$ A fuzzy left ideal $f$ of an
AG-groupoid $S$ is called fuzzy quasi-prime if for any two fuzzy left ideals 
$g$ and $h$ of $S$, $g\circ h\subseteq f$ implies that either $g\subseteq f$
or $h\subseteq f$. An AG-groupoid $S$ is fully fuzzy quasi-prime if every
fuzzy left ideal is quasi-prime in $S.$A fuzzy left ideal $f$ of an
AG-groupoid $S$ is called fuzzy semiprime left ideal of $S$ if for any fuzzy
left ideal $g$ of $S$, $g^{2}\subseteq f$ implies $g\subseteq f.$

The collection of fuzzy subsets $F(S)$ of an AG-groupoid $S$ is totally
ordered if for all fuzzy ideals $f,g$ of $S$ either $f\subseteq g$, or $%
g\subseteq f$.

\begin{theorem}
An AG-groupoid $S$ with left identity is fully fuzzy prime if and only if
every fuzzy ideal is idempotent and fuzzy ideals are totally ordered.
\end{theorem}

\begin{proof}
Let $S$ is fully fuzzy prime. Let $f$ is fuzzy ideal of $S,$ then since $%
f\circ f\subseteq S\circ f\subseteq f.$ By proposition 6 $f\circ f$ is fuzzy
ideal and by hypothesis $f\circ f$ is prime. So, $f\circ f\subseteq f\circ f$%
, implies that $f\circ f\subseteq f$. Consider, $f$ and $g$ be fuzzy ideals
of $S,$ then $f\circ g\subseteq f\circ S\subseteq f$ and also $f\circ
g\subseteq S\circ g\subseteq g.$ Hence, $f\circ g\subseteq f\cap g$, where $%
f\cap g$ is fuzzy ideal by lemma 2. By definition of fully fuzzy prime $%
f\subseteq f\cap g$ or $g\subseteq f\cap g,$ which implies that $f\subseteq
g $ or $g\subseteq f.$ Conversely, let every fuzzy ideal of $S$ is
idempotent and fuzzy ideals are totally ordered.\ Let $f$ be any fuzzy ideal
of $S$ such that $g\circ h\subseteq f$ where $f$ and $g$ are fuzzy ideals of 
$S.$ Since ideals are totally ordered so for $g$ and $h$ either $g\subseteq h
$ or $h\subseteq g.$ Let $g\subseteq h.$ Since $g$ is idempotent so, $%
g=g\circ g\subseteq g\circ h\subseteq f.$ Similarly, for $h\subseteq g$ we
have, $f\subseteq g.$
\end{proof}

\begin{proposition}
Let $S$ be an AG-groupoid with left identity. If $S$ is fully fuzzy
quasi-prime then every fuzzy left ideal is idempotent.
\end{proposition}

\begin{proof}
Let $f$ be a fuzzy left ideal in an AG-groupoid with left identity $S$,
where $S$ is fully fuzzy quasi-prime. Now, $S\circ f\subseteq f$ but $f\circ
f\subseteq S\circ f$ so $f\circ f\subseteq f.$ By proposition 6 $f\circ f$
is fuzzy ideal and by hypothesis $f\circ f$ is quasi-prime. So, $f\circ
f\subseteq f\circ f$, implies that $f\circ f\subseteq f$.
\end{proof}

\begin{theorem}
Let $S$ be an AG-groupoid with left identity. If $S$ is fully fuzzy
quasi-prime then for every fuzzy left ideals $f$ and $g$ of $S$, $f\circ
g=f\cap g.$
\end{theorem}

\begin{proof}
Let $f$ and $g$ are fuzzy left ideals in an AG-groupoid with left identity $%
S $, where $S$ is fully fuzzy quasi-prime. So, $S\circ f\subseteq f$ and $%
S\circ g\subseteq g,$ which implies that $g\circ f\subseteq f$ and $f\circ
g\subseteq g.$ By proposition 7 and theorem 1, $g\circ f=(g\circ g)\circ
\left( f\circ f\right) =\left( f\circ f\right) \circ \left( g\circ g\right)
=f\circ g,$ which implies that $f\circ g\subseteq f\cap g.$ Since $f\cap
g\subseteq f$ and $f\cap g\subseteq g$ so, $\left( f\cap g\right) \circ
\left( f\cap g\right) \subseteq f\circ g.$ By lemma 2 and proposition 7, $%
f\cap g$ is fuzzy left ideal and hence idempotent. So, $f\cap g\subseteq
f\circ g$ which implies that $f\cap g=f\circ g.$
\end{proof}

\begin{corollary}
The set of fuzzy quasi-prime ideals of AG-groupoid with left identity forms
a semillattice structure.
\end{corollary}

\begin{proof}
It is an easy consequence.
\end{proof}

\begin{theorem}
Let $S$ be an AG-groupoid with left identity, then followings are equivalent:

(i) each left ideal of $S$ is idempotent,

(ii) each fuzzy left ideal of $S$ is idempotent,

(iii) for each pair of fuzzy left ideals $f,\;g$ of $S,$ $f\circ g=f\cap g,$

(iv) each fuzzy left ideal of $S$ is fuzzy semiprime left ideal.
\end{theorem}

\begin{proof}
\textit{(i)}$\Leftrightarrow $\textit{(ii)} from [9].

Let \textit{(ii)} holds that is every fuzzy left ideal of $S$ is idempotent.
Let $f$ and $g$ are fuzzy left ideals which are idempotent, then $f\circ
g\subseteq g$ and $g\circ f\subseteq f.$ But $f\circ g=(f\circ f)\circ
g=\left( g\circ f\right) \circ f\subseteq f\circ f\subseteq f$ and so $%
f\circ g\subseteq f\cap g.$ Also $f\cap g\subseteq f$ and $f\cap g\subseteq
g $ which implies that $\left( f\cap g\right) \circ \left( f\cap g\right)
\subseteq f\circ g.$ We know that intersection of fuzzy left ideals $f$ and $%
g$ is fuzzy left ideal so $f\cap g$ is fuzzy left ideal of $S$ and hence an
idempotent, which implies that $f\cap g\subseteq f\circ g.$

\textit{(iii)}$\Rightarrow $\textit{(iv)} is same as in [9] using the
hypothesis for fuzzy left ideal $f$ in $S$ that $f\circ f=f\cap f=f.$

Let \textit{(iv) }holds that each fuzzy left ideal $f$ is semiprime in $S$.
By use of theorem 1 and lemma 4, $S\circ (f\circ f)=f\circ (S\circ f)=f\circ
f$ that is $f\circ f$ is fuzzy left ideal of $S$ and hence semiprime. Since $%
f\circ f\subseteq f\circ f$ which implies that $f\subseteq f\circ f$. But we
know that $f$ is fuzzy left ideal so $f\circ f\subseteq S\circ f\subseteq f.$
\end{proof}


\begin{thebibliography}{99}
\bibitem{Mordeson} Mordeson, J. N., \textit{Fuzzy semigroups},
Springer-Verlag Berlin Heidelberg, 2003.

\bibitem{N. Kuroki} Kuroki, N., On fuzzy ideals and fuzzy bi-ideals in
semigroups, Fuzzy Sets and Systems, 5 (1981) 203-215.

\bibitem{Qaiser M.} Mushtaq, Q. and Madad Khan, Ideals in AG-band and AG$%
^{\ast }$-groupoid, Quasigroups and Related Systems, 14 (2006), 207-215.

\bibitem{X. Y. Xie} Xiang-Yun Xie, On prime, quasi-prime, weakly quasi-prime
fuzzy left ideals of semigroups, Fuzzy Sets and Systems, 123 (2001), 239-249.

\bibitem{X. Y. Xie} Xiang-Yun Xie and Jian Tang, Fuzzy radicals and prime
fuzzy ideals of ordered semigroups, Information Sciences, 178 (2008),
4357-4374.

\bibitem{M.A. Kazim} Kazim, M. A. and M. Naseeruddin, On LA-Semigroups,
Allig. Bull. Math., 8 (1972), 1-7.

\bibitem{Protic} Protic, P. V. and N. Stevanovi\'{c}, On Abel-Grassmann's
groupoids, Proc. Conf. Pristina, 1994, 31-38.

\bibitem{Wang Xue} Wang Xue-ping, Mo Zhi-wen and Liu Wang-jin, Fuzzy Ideals
Generated by Fuzzy point in Semigroup, J. Sichuan Normal Univ., 15 (1992)
17-24.

\bibitem{Shabir, M.} Shabir, M., Fully fuzzy prime semigroups, Inter. J. of
Math. and Mathematical Sciences, 2005 (2005), 163-168.

\bibitem{L.A.Zadeh} Zadeh, L.A., Fuzzy sets, Inform. Control, 8 (1965),
338-353.

\bibitem{A. Rosenfelt} Rosenfeld, A., Fuzzy groups, J. Math. Anal. Appl., 35
(1971), 512-517.

\bibitem{N. Kuroki} Kuroki, N., Fuzzy bi-ideals in Semigroups, Comment.
Math. Univ. St. Pauli, 27 (1979), 17-21.
\end{thebibliography}
\end{document}